\documentclass[reqno,10pt]{article}
\usepackage{a4wide,color,eucal,enumerate,mathrsfs}
\usepackage[normalem]{ulem}
\usepackage{amsmath,amssymb,epsfig,bbm,graphicx}
\numberwithin{equation}{section}

\usepackage[latin1]{inputenc}

\newenvironment{proof}{\removelastskip\par\medskip   
\noindent{\em Proof.}
\rm}{\penalty-20\null\hfill$\square$\par\medbreak}

\newtheorem{theorem}{Theorem}[section]

\newtheorem{lemma}[theorem]{Lemma}
\newtheorem{proposition}[theorem]{Proposition}
\newtheorem{definition}[theorem]{Definition}

\newtheorem{remark}[theorem]{Remark}

\newcommand{\ep}{\varepsilon}
\newcommand{\Ent}{\mathcal{E}}

\def\ind{\mathbbm{1}}

\newcommand{\pical}{\mathcal{P}}
\newcommand{\dd}{\mathrm{d}}

\newcommand{\lcal}{\mathcal L}

\newcommand{\R}{\mathbb R}

\DeclareMathOperator{\Lip}{Lip}

\title{JKO estimates in linear and non-linear Fokker-Planck equations, and Keller-Segel: $L^p$ and Sobolev bounds}
\begin{document}
\author{
 Simone Di Marino\thanks{Universit\`a di Genova, Dipartimento di Matematica (DIMA); Via Dodecaneso 35, 16146 Genova, ITALIA, email: \textsf{simone.dimarino@unige.it}}, Filippo Santambrogio\thanks{Institut Camille Jordan, Universit\'e Claude Bernard - Lyon 1; 43 boulevard du 11 novembre 1918,
69622 Villeurbanne cedex,
FRANCE, email: \textsf{santambrogio@math.univ-lyon1.fr} }
   }

\maketitle

\abstract{We analyze some parabolic PDEs with different drift terms which are gradient flows in  the Wasserstein space and consider the corresponding discrete-in-time JKO scheme. We prove with optimal transport techniques how to control the $L^p$ and $L^\infty$ norms of the iterated solutions in terms of the previous norms, essentially recovering well-known results obtained on the continuous-in-time equations. Then we pass to higher order results, and in particulat to some specific BV and Sobolev estimates, where the JKO scheme together with the so-called ``five gradients inequality'' allows to recover some inequalities that can be deduced from the Bakry-Emery theory for diffusion operators, but also to obtain some novel ones, in particular for the Keller-Segel chemiotaxis model. }

\section{Short introduction}

The goal of this paper is to present some estimates on evolution PDEs in the space of probability densities which share two important features: they include a linear diffusion term, and they are gradient flows in the Wasserstein space $W_2$. These PDEs will be of the form
$$\partial_t\rho-\Delta\rho-\nabla\cdot (\rho\nabla u[\rho])=0,$$
complemented with no-flux boundary conditions and an intial condition on $\rho_0$.

We will in particular concentrate on the Fokker-Plack case, where $u[\rho]=V$ and $V$ is a fixed function (with possible regularity assumptions) independent of $\rho$, on the case where $u[\rho]=W*\rho$ is obtained by convolution and models interaction between particles, and on the parabolic-elliptic Keller-Segel case where $u[\rho]$ is related to $\rho$ via an elliptic equation. This last case models the evolution of a biological population $\rho$ subject to diffusion but attracted by the concentration of a chemo-attractant, a nutrient which is produced by the population itself, so that its distribution is ruled by a PDE where the density $\rho$ appears as a source term. Under the assumption that the production rate of this nutrient is much faster than the motion of the cells, we can assume that its distribution is ruled by a statical PDE with no explicit time-dependence, and gives rise to a system which is a gradient flow in the variable $\rho$ (the parabolic-parabolic case, where the time scale for the cells and for the nutrient are comparable, is also a gradient flow, in the product space $W_2\times L^2$, but we will not consider this case). Since we mainly concentrate on the case of bounded domains, in the Keller-Segel case the term $u[\rho]$ cannot be expressed as a convoluton and requires ad-hoc computations.

In all the paper, the estimates will be studied on a time-discretized version of these PDEs, consisting in the so-called JKO (Jordan-Kinderleherer-Otto) scheme, based on iterated optimization problems involving the Wasserstein distance $W_2$. We will first present 0-order estimates, on the $L^p$ and $L^\infty$ norms of the solution. This is just a translation into the JKO language of well-known properties of these equations. The main goal of this part is hence to popularize the techniques which allow to handle these estimates at a discrete level. However, there is an interest in studying these estimates at a discrete level, in particular since many numerical schemes have now been developed using the JKO approach, and these estimates can justify their convergence. 

Then, we will turn to 1st order estimates, i.e. on the gradient of the solutions. This includes in particular estimates on the BV norm of the solution $\rho$ and $W^{1,p}$-like estimates (in particular, the quantity that we will consider is related to $||\rho^{1/p}||_{W^{1,p}}$). We point out that a first result in this direction (estimates on the gradient for the JKO scheme) can be found in \cite{LeeJKO}, where the Lipschitz constant of the solution is bounded for the JKO scheme corresponding to a Fokker-Planck equation. However, the technique and the result in this paper are quite different than those in \cite{LeeJKO}.
 
The estimates we present are non-trivial and seem novel at least in the Keller-Segel case. In the Fokker-Planck case they correspond to a suitable integral version of the well-known Bakry-Emery estimate $|\nabla (P_t f)|\leq P_t (|\nabla f|)$ for drift-diffusion operators $P_t$ (see \cite{BakGenLed}). The interest in this case is to obtain them at a discrete level, on the JKO scheme. Note that, as the Bakry-Emery analogy suggests, these estimates should for sure be obtainable at a continuous level as well, but the computations are not at all easy (and most likely there is some term in the estimates which cannot easily be seen to have a sign, while the discrete-in-time approach allows to handle it without difficulties). This is an extra reason to study also the 0-order estimate at a discrete level, since some of these 1st order estimates require to use the corresponding 0-order ones. 

The JKO scheme provides, for fixed time step $\tau>0$, a sequence $(\rho^\tau_n)_n$, where each $\rho^\tau_{n+1}$ optimizes a functional depending on $\rho^\tau_n$. All the estimates that we provide are of the following form: a norm, or a quantity comparable to a norm, computed at $\rho^\tau_{n+1}$ can be bounded in terms of the same expression computed at $\rho^\tau_n$. Of course, we only want estimates which can be iterated (i.e. the possible increase passing from $\rho^\tau_{n}$ to $\rho^\tau_{n+1}$ should be of the order of $\tau$) and which do not explode when $\tau\to 0$. When a same quantity is really decreasing along iterations - in particular if an exponential decreasing behavior is obtained - this can be used to study the asymptotic behavior of the solution $\rho_t$ of the PDE as $t\to\infty$. When there is no decreasing behavior, but the increase is controlled, this can be used to justify local-in-time bounds which can provide compactness (to be used either for the convergence of numerical schemes or for other stability results, when data are varying, for instance).

The paper and the results are organized as follows. After this introduction, we present in Section 2 the background that we need to use about the JKO scheme for gradient flows in the Wasserstein space, including some useful tools such as displacement convexity and the five-gradients-inequality, together with general facts on optimal transportation and some details on the functionals that we will use. Section 3 presents the main estimates on the $L^p$ and $L^\infty$ norms of the solution of one step of the JKO scheme in the case where the functional is either a potential energy $\rho\mapsto\int V d\rho$ or an interaction energy $\rho\mapsto \frac 12 \int W(x-y) d\rho(x)d\rho(y)$. In particular we prove iterable bounds on the $L^p$ norm, for $p<\infty$, when $V$ or $W$ are Lipschitz, as well as better bounds (which include an $L^\infty$ estimate and an $L^p$ one which can be used in the limit $p\to\infty$ and also provide an exponential $L^\infty$ bound) in the case of the potential energy under second-order condition on $V$. As far as the $L^\infty$ norm is concerned, we also provide a uniform bound stating that the maximal value of $\rho e^V$ is decreasing in time under essentially no assumption on $V$, together with an adaptation for the interaction case, when $W$ is Lipschitz continuous. These results are summarized in Proposition \ref{summary warmup}. Section 4 concentrates, then, on the Keller-Segel case,  and reproduces, in this discretized JKO setting, a well-known two-dimensional result (based on \cite{DolPer} and \cite{JL}) which states that the $L^p$ norm and the $L^\infty$ norms do not grow too much in time as soon as we are in the subcritical regime, an assumption which allows to control the entropy with the total energy itself. This very technical result is contained in Theorem \ref{lpestimate}. Finally, Section 5 is devoted to higher-order estimates, which are the core of the paper. The results are expressed in terms of the following quantity: given a convex function $H:\R^d\to\R$, we consider $\int H(Z_\rho)d\rho$, where 
$$Z_\rho:=\frac{\nabla\rho}{\rho}+\nabla u[\rho].$$ 
When $H(z)=|z|^p$ and $\nabla u[\rho]$ is bounded, this quantity (usually denoted by $J_{(p)}(\rho)$), is comparable to $\int |\nabla \rho|^p \rho^{1-p}dx$, which can be related by simple algebraic computations to the $W^{1,p}$ norm of $\rho^{1/p}$. On these quantities we prove iterable bounds in the case of the potential energy when $V$ is semi-convex (Proposition \ref{prop5.3}); it is also useful to consider other convex functions $H$ than only powers, which can provide Lipschitz bounds and $W^{1,1}$ regularity. A variant of this result exists for interaction energies, possibly combined with potential energies (Proposition \ref{decrHZ}, where we assume semiconvexity of $V$ and $C^{1,1}$ regularity for $W$). The results for the potential and interaction cases are contained in the sub-section 5.1, while the sub-section 5.2 is devoted to the Keller-Segel case. In this case, the lack of semiconcavity for $u[\rho]$, which is only defined as a solution of an elliptic PDE involving $\rho$, prevents from having easy estimates on the error terms, and a different technique is required to bound them: finally, we obtain an iterable estimate on $J_{(p)}(\rho)$ only for $p<2$, and under the extra assumption that $\rho$ is bounded in an $L^r$ space, with $r=(4-p)/(2-p)$ depending on $p$ and explosing as $p\to 2$. This explains the interest for the 0-order estimates at the JKO level for Keller-Segel, which can indeed guarantee such an $L^r$ assumption.

\section{Preliminaries on the JKO scheme}

We refer to \cite{AmbGigSav,OTAM,GradFlowSurvey} for the whole theory about gradient flows in the Wasserstin space which justifies the few facts that we list below.

Whenever a functional $\mathcal F:\pical(\Omega)\to\R\cup\{+\infty\}$ is given, we fix a time step $\tau>0$ and a measure $\eta\in\pical(\Omega)$, and consider the following minimization problem
\begin{equation}\label{JKObasic}\min_\rho\quad \mathcal F(\rho)+\frac{W_2^2(\rho,\eta)}{2\tau},\end{equation}
where $W_2$ is the Wasserstein distance of order $2$ (see \cite{AmbGigSav,villani,OTAM}). Before goind on with the discussion, let us remind few important facts about optimal transport and the $W_2$ distance.

 If two probabilities $\mu,\nu\in\pical(\Omega)$ are given on a compact domain, the Monge-Kantorovitch problem reads as
$$ \inf\big\{ \int |x-T(x)|^2 d\mu\;:\;T:\Omega\to\Omega,\;T_\#\mu=\nu\big\}.$$
This problem, introduced by Monge \cite{Monge} has been reformulated by Kantorovich, \cite{Kant} in the following convex form
$$\inf\big\{ \int \!|x-y|^2 d\gamma\,:\,\gamma\in\pical(\Omega\times\Omega),\,(\pi_x)_\#\gamma=\mu, (\pi_y)_\#\gamma=\nu\big\}.$$
The square root of the optimal value above defines a distance on the set of probability measures on a given compact space (in case of non-compactness a condition on the second moments has to be added), which by the way metrizes the weak-* convergence of probabilities (on compact spaces, without compactness there is again a condition on the moments). Kantorovich also provided a dual formulation for the above minimization problem, that we can state, for simplicity, using the cost function $|x-y|^2/2$:
 $$
\frac12 W_2^2(\mu,\nu)=\sup\big\{ \int\varphi\,d\mu+\int\psi\,d\nu\;:\;\varphi(x)+\psi(y)\leq\frac 12|x-y|^2\big\}.
$$
It is possible to prove the existence of an optimal $\gamma$ and of an optimal pair $(\varphi,\psi)$, and, as soon as $\mu$ is absolutely continuous, there exists as well an optimal transport map $T$ (and the optimal $\gamma$ will be a measure on $\Omega\times\Omega$ concentrated on the graph of such a map $T$) . Moreover, the optimal $\varphi$, called {\it Kantorovich potential,} is Lipschitz continuous. and is connected to the optimal $T$ via $T(x)=x-\nabla\varphi(x)$ (we can also write $T=\nabla u$ with $u(x)=|x|^2/2-\varphi(x)$, and $u$ is a convex function, which is the result of the celebrated Brenier's Theorem, \cite{Brenier polar,Brenier91}).

 Using these tools from optimal transport theory, if $\Omega$ is compact and $\mathcal F$ is l.s.c. for the weak convergence of probability measures, then Problem \eqref{JKObasic} admits at least a solution. We will denote the set of solutions as
$Prox_{ \mathcal F}^{\tau} (\eta)$, mimicking the notations for the proximal operator which are used in hilbertian settings. In some cases (in particular if $\mathcal F$ is stricty convex) this proximal operator is single-valued (i.e. the minimizer is unique), but this will not be crucial in our analysis.

The JKO scheme (introduced in \cite{JKO}) consists in iterating the above minimization problem, i.e.starting from $\rho_0$ and, for fixed $\tau>0$, defining a sequence $(\rho^\tau_n)_n$ satisfying
$$\rho^\tau_0=\rho_0,\quad \rho^\tau_{n+1}\in Prox_{ \mathcal F}^{\tau} (\rho^\tau_n).$$

The above sequence can be used to define a curve of measure $\rho^\tau(t)$ for $t\in [0,T]$, with $\rho^\tau(n\tau)=\rho^\tau_n$ (for instance by piecewise constant interpolation). Under suitable conditions on $\mathcal F$ it can be proven that he curves $\rho^\tau$ uniformly converge (as curves valued into the Wasserstein space) to a continuous curve $\rho$ which is a solution of the PDE  
$$\partial_t\rho-\Delta\rho-\nabla\cdot (\rho\nabla \frac{\delta \mathcal F}{\delta \rho})=0$$
(complemented with no-flux boundary conditions and the intial condition $\rho_0$), where $\frac{\delta \mathcal F}{\delta \rho}$ is the first variation of the functional $\mathcal F$ (see Chapter 7 in \cite{OTAM}).

In this paper we will always consider the case where $\mathcal F=\Ent+\mathcal G$, and 
$$\Ent(\rho)=\begin{cases}\int \rho\log\rho \,dx& \mbox{ if $\rho\ll\lcal^d$}\\
						+\infty&\mbox{ otherwise}\end{cases}$$
is the entropy functional. For the functional $\mathcal G$, we will often write $u[\rho]=\delta \mathcal G/\delta \rho$ and we will consider three cases:
\begin{itemize}
\item either we consider $\mathcal G(\rho)=\int Vd\rho$, for a fixed function $V:\Omega\to\R$ acting as a potential, in which case we have $u[\rho]=V$; this case will be called the Fokker-Planck case;
\item either we consider $\mathcal G(\rho)=\int W(x-y)d\rho(x)d\rho(y)$, for an even function $W:\R^d\to\R$, in which case we have $u[\rho]=W*\rho$; this case will be called the interaction case; it can be mixed with the previous one by considering $u[\rho]=V+W*\rho$, if explicitly indicated;
\item finally, we consider a particular case arising from mathematical biology: we take $\mathcal G(\rho):=-\frac{\chi}{2} \int |\nabla h[\rho]|^2dx=-\frac{\chi}{2} \int h[\rho] d\rho,$
where $\chi>0$ is a given constant and $h[\rho]$ is the only solution of 
$$\begin{cases}-\Delta h=\rho &\mbox{in } \Omega,\\
			h=0 &\mbox{on } \partial\Omega\end{cases}.$$
Note the negative sign before the integral in the definition of $\mathcal G$.
It is not difficult to check that we have
$$\frac{\delta\mathcal G}{\delta \rho}=-\chi h[\rho].$$
Indeed $h[\rho+\ep\delta\rho]=h[\rho]+\ep h[\delta\rho]$ and 
$$\mathcal G(\rho+\ep\delta\rho)=\mathcal G(\rho)-\ep\chi \int \nabla h[\rho]\cdot \nabla h[\delta\rho]dx+O(\ep^2)=\mathcal G(\rho)+\ep\chi\int h[\rho] \Delta h[\delta\rho]dx+O(\ep^2),$$
which allows to conclude using $ \Delta h[\delta\rho]=-\delta\rho$.
This case will be called the Keller-Segel case and is motivated by chemotaxis modeling (see \cite{KS,HP} for the description of the model). In dimension $d=2$, it is well-known that this model is well-posed and that there is existence (both for the minimization problems in the JKO scheme and for the continuous-in-time PDE, with global-in time existence) as soon as $\chi<8\pi$. This  is due to the a crucial inequality which states that we can bound $\Ent(\rho)$ in terms of $\Ent(\rho)+\mathcal G(\rho)$ (the problem being that $\mathcal G$ is in general not bounded from below, but $\Ent+\mathcal G$ is bounded from below on probability measures as soon as $\chi\leq 8\pi$: this implies 
$$\Ent(\rho)\leq A+B(\Ent(\rho)+\mathcal G(\rho))$$
with $B=8\pi/(8\pi-\chi)$). For the mathematical analysis of the Keller-Segel PDE and of the corresponding JKO scheme we refer to \cite{BlaCalCar,BCKKLL,BlaCarMas,CC,CarLisMai} and to Chapter 5 in \cite{Per-book}. \end{itemize}

The reader may need to be convinced of the bound from below of $\Ent(\rho)+\mathcal G(\rho)$ when $\chi\leq8\pi$ in dimension 2, if we are on a bounded domain and $h[\rho]$ is defined with Dirichlet boundary conditions on $\partial\Omega$. This can be seen by observing the following facts. The logarithmic Hardy-Littlewood-Sobolev inequality provides a uniform bound from below on $\int \rho\log\rho \,dx-4\pi\int \tilde h[\rho]\, d\rho$ where $\tilde h[\rho](x):=-(2\pi)^{-1}\int_{\R^2} \log(|x-y|)d\rho(y)$. Noting that we have $-\Delta \tilde h[\rho]=\rho$ and $\tilde h[\rho]+\log(R)/(2\pi)\geq 0$ on $\Omega$ (where $R$ is the diameter of $\Omega$), we deduce $h[\rho]\leq \tilde h[\rho]+\log(R)/(2\pi)$ (since $\tilde h[\rho]+\log(R)/(2\pi)-h[\rho]$ is harmonic and nonnegative on the boundary). 
Hence, for $\chi\leq 8\pi$ we have $\Ent(\rho)+\mathcal G(\rho)\geq \int \rho\log\rho \,dx-4\pi\int h[\rho]\, d\rho\geq \int \rho\log\rho \,dx-4\pi\int \tilde h[\rho]\, d\rho -2\log R$ and this provides the desired bound from below.
\bigskip

A useful tool, introduced in \cite{demesave} and already used in the framework of the JKO scheme in the same paper in order to obtain BV estimates is the so-called {\it five-gradients inequality} (note that this name is not present in \cite{demesave}, but the inequality has been popularized under this name later on). This inequality states the following:

\begin{lemma}\label{5GI}
Let $\Omega\subset\R^d$ be bounded and convex, $\rho,\eta\in W^{1,1}(\Omega)$ be two  probability densities and $H\in C^1(\R^d)$ be a radially symmetric convex function. 
Then the following inequality holds
\begin{equation}\label{maine1}
\int_{\Omega} \Big(\nabla \rho\cdot\nabla H(\nabla \varphi) +\nabla \eta\cdot \nabla H(\nabla \psi)\Big)\,\dd x\ge 0,
\end{equation}
where $\varphi$ and $\psi$ are the corresponding Kantorovich potentials.
\end{lemma}
Note that the above result is first proven for $H\in C^2$ (second derivatives are used in the proof) and then, by approximation, it stays true for $H\in C^1$; the same approximation can also be applied to the quite common case $H\in C^1(\R^d\setminus\{0\})$, setting $\nabla H(0):=0$ (which is coherent with the fact that $H$ is radial), and the result stays true. In particular, we will sometimes apply this to $H(z)=|z|$.

Another useful notion in the study of gradient flow is that of displacement convexity, introduced by McCann in \cite{MC}. It corresponds to the convexity of a functional along the geodesics of the metric space $(\pical(\Omega),W_2)$. 

\begin{definition} Let $\mathcal H:\pical(\Omega)\to\R\cup\{+\infty\}$ be a functional defined on probability measures on a compact convex domain $\Omega$. We say that $\mathcal H$ is displacement convex if for every pair of measures $\rho,\eta\in \pical(\Omega)$ there exists a curve $\rho_t$ which is geodesic for the $W_2$ distance, which connects $\rho$ and $\nu$ (i.e. $\rho_0=\rho$, $\rho_1=\nu$) and such that $\mathcal H(\rho_t)\leq (1-t)\mathcal H(\rho_0)+t\mathcal H(\nu)$.
\end{definition}
We recall that, whenever $\rho$ is absolutely continuous, the geodesic curve between $\rho$ and $\nu$ is unique and is given by
$$\rho_t=(id-t\nabla \varphi)_\#\rho,$$
where $\varphi$ is the Kantorovich potential between $\rho$ and $\nu$ for the cost $c(x,y)=\frac 12 |x-y|^2$. Indeed, $id-t\nabla \varphi$ is the convex interpolation between the identity map and the optimal transport map $T=id-\nabla\varphi$. 

In \cite{MC} McCann provided the condition for the displacement convexity of functionals of the form $\mathcal H(\rho):=\int F(\rho(x))dx$.

\begin{definition} Let $F$ be a convex increasing function on $[0,+\infty)$ such that $F(0)=0$. Then we say that $F$ satisfies the $d$-McCann condition if $s \mapsto F(\frac 1{s^d})s^d $ is convex and decreasing.
\end{definition}

Note that $s \mapsto F(\frac 1{s^d})s^d $ being convex and decreasing is enough to guarantee that $F$ itself is convex.

The main result of \cite{MC} is indeed the fact that, if $F$ satisfies the $d$-McCann condition, then the functional $\mathcal H$, defined via $\mathcal H(\rho):=\int F(\rho(x))dx$, is displacement convex in dimension $d$. In particular, this applies to $F(s)=s^q$, $q>1$, and to $F(s)=s\log s$ (hence to $\mathcal H=\Ent $).

In \cite{AmbGigSav} the general theory for gradient flows in metric space is presented, and the assumption of geodesic convexity is crucial, in particular for uniqueness and stability. Here we do not insist on this aspect (by the way, the functional $\mathcal G$ in the Keller-Segel case is in general not displacement convex), but we are interested in another property related to displacement convexity. As it was first observed in \cite{MatMcCSav}, estimates can be provided on $\mathcal H(\rho^\tau_{n+1})$ in terms of $\mathcal H(\rho^\tau_{n})$ when $\mathcal H$ is displacement convex, even when the gradient flow that we are considering is the gradient flow of another functional $\mathcal F$ (in the case $\mathcal F= \mathcal H$, the inequality  $\mathcal H(\rho^\tau_{n+1})\leq \mathcal H(\rho^\tau_{n})$ is trivial). The key point is to use the following general estimate.

\begin{lemma}\label{estfromdisplconv} Let us consider two absolutely continuous measures $\rho, \eta \in \mathscr{P}(\Omega)$, and a convex function $F$, such that $F(0)=0$ satisfying the $d$-McCann condition. Suppose that the density of $\rho$ is Lischitz continuous, and that $\Omega$ is convex. Then, denoting by $\varphi$ the Kantorovich potential in the transport from $\rho $ to $\eta$, we have
\begin{equation}\label{eqn:F} \int_{\Omega} F( \eta) \, dx \geq \int_{\Omega} F(\rho) \, dx - \int_{\Omega} \rho \nabla(F'(\rho))\cdot \nabla \varphi \, dx.\end{equation}


\end{lemma}

\begin{proof} Using the displacement convexity of $\mathcal H$, we have 
$$\mathcal H(\eta)-\mathcal H(\rho)=\mathcal H(\rho_1)-\mathcal H(\rho_0)\geq \frac{d}{dt}\mathcal H(\rho_t)_{|t=0},$$
where $(\rho_t)_t$ is the geodesic interpolation between $\eta=\rho_1$ and $\rho=\rho_0$. 
We just need to prove that we have 
\begin{equation}\label{dHdt}\frac{d}{dt}\mathcal H(\rho_t)_{|t=0}\geq  - \int_{\Omega} \rho \nabla(F'(\rho))\cdot \nabla \varphi \, dx.
\end{equation}
A formal computation gives
\begin{equation}\label{dHdt2}\frac{d}{dt}\mathcal H(\rho_t)=\int F'(\rho_t)\partial_t\rho_t=\int \nabla(F'(\rho_t))\cdot v_td\rho_t,\end{equation}
where $v_t$ is the velocity field of the geodesic curve $(\rho_t)_t$, solving $\partial_t\rho_T+\nabla\cdot (\rho_tv_t)=0$. The equality $v_0=-\nabla\varphi$ provides the result.

The reader should be aware that this argument is only formal because of lack of regularity. Yet, everything could be justified by a precise computation of the density of the measure $\rho_t$. This is classical but technicaly delicate, and it is done, for instance, in Appendix A2 in \cite{BlaMosSan}. In particular, the possible presence of singular parts in the second derivatives of $\varphi$ justifies the inequality in \eqref{dHdt}, instead of the equality we found using \eqref{dHdt2}. \end{proof}

Then, this can provide estimates on $\mathcal F(\rho)$ once we suppose $\rho\in Prox_{ \mathcal F}^{\tau} (\eta)$ and use the optimality condition in the optimization problem solved by $\rho$, which is of the form
$$\frac\varphi\tau+\frac{\delta F}{\delta \rho}=const \mbox{ on }\{\rho>0\},$$
(see Chapter 7 in \cite{OTAM} for precise statements and justifications on these optimality conditions). The consequence in the case $\mathcal F=\Ent+\mathcal G$ is presented in the next section.

\section{Warm-up: $L^p$ and $L^\infty$ estimates for Fokker-Planck and interaction equations}

In this section we present various computations leading to $L^p$ estimates (including $p=\infty$) for the simplest case that we consider, i.e. the linear Fokker-Planck case with $\mathcal G(\rho)=\int Vd\rho$. We will then adapt them to the case where the first variation $u[\rho]$ depends on $\rho$ (while for the Fokker-Planck case we do have $u[\rho]=V$ for every $\rho$) but in a very simple way, by convolution.

We start from the following result.

\begin{proposition}\label{F''} Let $\eta$ be a probability measure, $\Omega\subset\R^d$ a convex domain, and $F\in C([0,\infty))\cap C^2((0,\infty))$ be a convex function satisfying the $d-$McCann condition. Let $\mathcal G:\pical(\Omega)\to\overline\R$ be a given functional, $\rho \in Prox_{ \Ent + \mathcal G}^{\tau} (\eta)$, and $u[\rho]:=\delta\mathcal G/\delta\rho$. Suppose that $u[\rho]$ is Lipschitz continuous. Then $\rho$ is also Lipschitz continuous, and bounded from below by a positive constant, and, if $\Omega$ is convex, we have
\begin{equation}\label{eqn:stimaKS} \int_{\Omega} F(\eta) \, dx \geq \int_{\Omega} F(\rho) \, dx + \tau \int_{\Omega} \left(F''(\rho)| \nabla\rho|^2+ \rho F''(\rho) \nabla\rho\cdot\nabla u[\rho])\right). \end{equation}

\end{proposition}

\begin{proof}
This estimate is a combination of the one in Lemma \ref{estfromdisplconv} with the optimality conditions characterizing $\rho$. Indeed, we have (see Chapter 8 in \cite{OTAM} and adapt the computations which are just presented there in the case $u[\rho]=V$) 
$$\log\rho+u[\rho]+\frac\varphi\tau=const\quad\mbox{hence}\quad \frac{\nabla\rho}{\rho}+\nabla u[\rho]+\frac{\nabla\varphi}{\tau}=0,$$
these equalities being true a.e. on $\Omega$ since we have $\rho>0$ a.e. (for this, see the proof in  Chapter 8 in \cite{OTAM}). As a consequence, $\log\rho$ is Lipschitz continuous, and we can apply the result of Lemma \ref{estfromdisplconv}, replacing $\nabla\varphi$ with $-\tau(\nabla u[\rho]-\nabla\rho/\rho)$.
\end{proof}

Let us analyze first the purely linear Fokker-Planck case, i.e. the case where $u[\rho]=V$ does not depend on $\rho$ and let us concentrate on $L^p$ estimates.

\begin{proposition}\label{Lp on truly rho} Let $\eta\in L^p$, with $p<\infty$, be a probability measure and $\mathcal G:\pical(\Omega)\to\R$ be defined via $\mathcal G(\rho):=\int Vd\rho$ for a given function $V:\Omega\to\R$. Take $\rho \in Prox_{ \Ent + \mathcal G}^{\tau} (\eta)$ and suppose $\Omega$ convex and $V$ Lipschitz continuous.
Then $\rho $ is Lipschitz continuous and bounded from below by a positive constant, and
\begin{itemize}

\item denoting by $\Lip(V)$ the Lipschitz constant of $V$, we have
$$\int \eta^p \,dx\geq  \left(1-\tau\frac{p(p-1)}{4}\Lip (V)^2\right)\int \rho^p\,dx;$$
\item if $\Delta V\leq A$ in $\Omega$ and $\nabla V\cdot n\geq 0$ on $\partial\Omega$, then
$$\int \eta^p \,dx\geq \left(1-\tau(p-1)A\right) \int\rho^p\,dx.$$
\end{itemize}

In the case $p=\infty$ the result is the following: if $\Delta V\leq A$ in $\Omega$ and $\nabla V\cdot n\geq 0$ on $\partial\Omega$, then
$$||\rho||_{\infty}\leq ||\eta||_{\infty}\left(1+\tau\frac{A}{d}\right)^d.$$
\end{proposition}

\begin{proof}
The estimates on the $L^p$ norm are consequences of the estimate in Proposition \ref{F''}, applied to $F(s)=s^p$. In this case we obtain
$$
 \int_{\Omega} \eta^p \, dx \geq \int_{\Omega} \rho^p \, dx + \tau p(p-1) \left(\int_{\Omega} \rho^{p-2}| \nabla\rho|^2+ \rho^{p-1}  \nabla\rho\cdot\nabla V\,dx)\right). $$
 For the first estimate, we apply a Young inequality which gives
 $$\int \rho^{p-2}\nabla\rho\cdot (\rho\nabla V)\,dx\geq -\int \rho^{p-2}|\nabla\rho|^2-\frac 14\int \rho^{p}|\nabla V|^2\,dx,$$
 which proves the claim.
 
 For the second, we ignore the positive term $\int_{\Omega} \rho^{p-2}| \nabla\rho|^2$ and we rewrite the remaining as
 $$ \tau p(p-1)\int_\Omega \rho^{p-1}  \nabla\rho\cdot\nabla V\,dx=\tau(p-1) \int_\Omega \nabla (\rho^{p} ) \cdot\nabla V\,dx$$
and we integrate by parts, using our assumptions on $V$.

The last statement, about the $L^\infty$ norm, can be proven proven differently, following the same strategy as in Proposition 7.32 in \cite{OTAM}. The adaptations to be performed in the proof are the following: instead of a minimum point of $\varphi$ we take a minimum point for $\varphi+\tau V$ ; we first use $\det(I-D^2\varphi)\leq (1-\Delta \varphi/d)^d$, which is a consequence of the geometric-arithmetic mean inequality, and then $-\Delta \varphi\leq \tau\Delta V\leq A$. The assumption on $\partial V/\partial n$ is needed to handle a possible maximizer on the boundary. This statement can be first proven for $V\in C^2$ and then by approximation, for a less regular $V$, where the assumptions on the Laplacian and on the sign of the normal derivaitve should be interpreted as a condition on the distributional divergence of the vector field $\nabla V$, extended as $0$ outside $\Omega$.
\end{proof}

\begin{remark}\label{LpLinfty}
We observe that, in the continous-time limit, the second estimate gives $\int \rho_t^pdx\leq e^{(p-1)At}\int \rho_0^pdx$. It is then possible to raise to power $1/p$ and take the limit $p\to \infty$ and obtain $||\rho_t||_{\infty}\leq e^{At}||\rho_0||_{\infty}$. Unfortunately, this cannot be done in discrete time since if we first send $p\to\infty$ for fixed $\tau>0$ the coefficient in the r.h.s. in front of $\int \rho^p$ can become negative. This is why we presented a different technique for the $L^\infty$ estimate, but we can notice that, asymptotically as $\tau\to 0$, the the two results coincide.

Also, we observe that the first estimate is not suitable for a limit $p\to\infty$, as the coefficient in the r.h.s. is quadratic in $p$, so that its continuous-in-time version is $\int \rho_t^pdx\leq e^{p(p-1)\Lip(V)^2t/4}\int \rho_0^pdx$.
\end{remark}

\begin{proposition}\label{exponentialLinfty} Let $\eta\in L^\infty$, be a probability measure $\mathcal G:\pical(\Omega)\to\R$ be defined via $\mathcal G(\rho):=\int Vd\rho$ for a given bounded function $V:\Omega\to\R$. Take $\rho \in Prox_{ \Ent + \mathcal G}^{\tau} (\eta)$ and suppose $\Omega$ convex.
Then $\rho $ is bounded, and satisfies
$$||\rho e^V||_{\infty} \leq||\eta e^V||_{\infty}.$$
\end{proposition}

\begin{proof}
The proof, consisting in looking at the maximum point of $\rho e^V$, i.e. of $\log\rho+V$, is exactly the same as in Lemma 2.4 of \cite{IacPatSan}. Indeed, we can write 
$$\Ent(\rho)+\int Vd\rho=\int (\rho\log\rho+\rho V)dx=\int F\left(\frac{\rho}{m}\right)d m,$$
for $F(s)=s\log s$ and $m=e^{-V}$.
\end{proof}

\begin{remark}
Actually, as it is proven in \cite{IacPatSan}, the same bound also holds from below. Moreover, morally this upper bound weighted by $e^V$ should also work for the $L^p$ estimates, but it would require the geodesic convexity of the $L^p$ norm w.r.t. $e^{-V}$, which requires $V$ to be convex (and if $V$ is only $\lambda$-convex, it does not work). It is interesting to see that the assumption is the opposite of the one in Proposition \ref{Lp on truly rho}, where we needed upper bounds on $D^2V$. 
\end{remark}

The case where $u[\rho]$ depends on $\rho$, but in a very good way, is easy to handle. Take an even function $W:\R^d\to\R$ and consider
$$u[\rho]=W*\rho,\quad\mbox{i.e.}\quad\mathcal G(\rho)=\frac 12\int W(x-y)\,d\rho(x)\,d\rho(y).$$

The first of the two estimates in Proposition \ref{Lp on truly rho} is easy to adapt, while unfortunately the other one, based on second-order assumptions but also on the boundary behavior, cannot be easily translated in terms of $W$. The same problem occurrs for the $L^\infty$ estimate of Proposition \ref{Lp on truly rho}. We can therefore state the following:
\begin{proposition} Let $\eta\in L^p$, with $p<\infty$, be a probability measure $\mathcal G:\pical(\Omega)\to\R$ be defined via $\mathcal G(\rho):=\frac 12\int W(x-y)\,d\rho(x)\,d\rho(y)$ for a given even function $W:\R^d\to\R$. Take $\rho \in Prox_{ \Ent + \mathcal G}^{\tau} (\eta)$ and suppose $\Omega$ convex and $W$ Lipschitz continuous.
Then $\rho $ is Lipschitz continuous and bounded from below by a positive constant, and
$$\int \eta^p \,dx\geq  \left(1-\tau\frac{p(p-1)}{4}\Lip(W)^2\right)\int \rho^p\,dx.$$
\end{proposition}
\begin{proof}
The proof is identical to that in Proposition \ref{Lp on truly rho}, which does not depend on the fact that $u[\rho]$ depends or not upon $\rho$, but only on its Lipschitz bounds. Hence, we just have to observe that we have $\Lip(W*\rho)\leq \Lip(W)$.
\end{proof}

On the other hand, it is easier to extend the estimate in Proposition \ref{exponentialLinfty}, but this requires an adaptation if one wants to iterate it.
\begin{proposition} Let $\eta\in L^\infty$ be a probability measure and $\mathcal G:\pical(\Omega)\to\R$ be defined via $\mathcal G(\rho):=\frac 12\int W(x-y)\,d\rho(x)\,d\rho(y)$ for a given even function $W:\R^d\to\R$. Take $\rho \in Prox_{ \Ent + \mathcal G}^{\tau} (\eta)$ and suppose $\Omega$ convex and $W$ Lipschitz continuous.
Then $\rho $ is Lipschitz continuous and bounded from below by a positive constant, and
$$||\rho e^{u[\rho]}||_{\infty} \leq||\eta e^{u[\rho]}||_{\infty}.$$
This implies in particular the (more useful estimate)
$$||\rho e^{u[\rho]}||_{\infty}e^{\Ent(\rho) + \mathcal G(\rho)}\leq ||\eta e^{u[\eta]}||_{\infty}e^{\Ent(\eta) + \mathcal G(\eta)}e^{\tau\Lip(W)^2/2}.$$
\end{proposition}
\begin{proof}
The estimate $||\rho e^{u[\rho]}||_{\infty} \leq||\eta e^{u[\rho]}||_{\infty}$ can be trivially obtained in the same way as in the case where $u[\rho]=V$ does not depend on $\rho$. Then we observe that we have
$$u[\rho]\leq u[\eta]+||W*(\rho-\eta)||_{\infty}\leq u[\eta]+\Lip(W)W_1(\rho,\eta).$$
We then use
$$\Lip(W)W_1(\rho,\eta)\leq \Lip(W)W_2(\rho,\eta)\leq \frac \tau 2 \Lip(W)^2+\frac{W_2^2(\rho,\eta)}{2\tau}\leq \frac\tau 2 \Lip(W)^2+\mathcal F(\eta)-\mathcal F(\rho), $$
where $\mathcal F=\Ent+\mathcal G$. This provides the claimed result.
\end{proof}

We can now deduce, frome the various estimates of this section, the following bounds, whose proofs are just a combination of the arguments above.
\begin{proposition}\label{summary warmup}
Suppose $\rho_0\in L^p$ and consider the JKO scheme for the functional $\mathcal F=\Ent+\mathcal G$ with $\mathcal G(\rho)=\int Vd\rho$. Then we have
\begin{itemize}
\item if $p<+\infty$ and $V$ is Lipschitz continuous, then the norm $||\rho_t||_p$ grows at most exponentially in time;
\item if $p\in [1,+\infty]$ is arbitrary and $V$ is such that $\Delta V$ is bounded from above and $\nabla V\cdot n\geq 0$ on $\partial\Omega$, then  $||\rho_t||_p$ grows at most exponentially in time.
\item if $p=\infty$ and $V$ is bounded, then the norm  $||\rho_te^V||_\infty$ is non-increasing in time. 
\end{itemize}
In the case of the interaction functional $\mathcal G(\rho)=\frac 12\int W(x-y)\,d\rho(x)\,d\rho(y)$, we have
\begin{itemize}
\item if $ p<+\infty$ and $W$ is Lipschitz continuous, then the norm $||\rho_t||_p$ grows at most exponentially in time;
\item  if $p=\infty$ and $W$ is Lipschitz continuous,  then the quantity
$$||\rho_t e^{W*\rho_t}||_{\infty}e^{ \mathcal F(\rho_t)}$$
 is non-increasing in time, and in particular the norm $||\rho ||_{\infty}$ is uniformly bounded.
\end{itemize}
\end{proposition}

\section{$L^p$ and $L^\infty$ estimates for the Keller-Segel case}

\begin{proposition} The following functions are convex and satisfy the $d$-McCann condition:
\begin{itemize}
\item $F_{p,K}(s)= (s^p - K s^{\frac {d-1}d} )_+$ for every $p\geq 1$ and $K>0$.
\item $\tilde F_{p,K}(s) = (s-K)_+^p$ for $p \geq \frac {4d}{3d+1}$ and every $K>0$;
\end{itemize}
\end{proposition}
We will mainly use the functions $F_{p,K}(s)$, so as to avoid restrictions on $p$, but it is quite apparent that computations are easier with $\tilde F_{p,K}(s)$, which is the reason for presenting both.

\begin{proof} It is clear that both $F_p$ and $\tilde F_p$ are convex and increasing. In the case of $F_p$ it is sufficient to compute 
$$F_p \left( \frac 1{s^d} \right) s^d = \left( \frac 1{s^{d(p-1)}} - Ks \right)_+,$$
which is clearly convex. For the functions $\tilde F_{p,K}(s)$,  we need to compute the derivatives:

$$ \frac {d}{ds} \Bigl( \tilde F_p ( s^{-d} ) s^d \Bigr) = d s^{d-1} \left( \frac 1 {s^d} \!-\! K \right)_+^p - \frac{pd}s \left( \frac 1{s^d} \!-\! K \right)_+^{p-1} \!\!=- d\left( \frac {p-1}s + Ks^{d-1}\! \right) \cdot  \left( \frac 1{s^d} - K \right)_+^{p-1}$$
\begin{align*} \frac{\frac {d^2}{ds^2} \Bigl( \tilde F_p ( s^{-d} ) s^d \Bigr)}{ d(s^{-d} - K)_+^{p-2}} & =  \left( \frac{ p-1}{s^2} - K(d-1)s^{d-1}\right) \cdot  \left( \frac 1 {s^d} \!-\! K \right) + \left( \frac {p-1}s + Ks^{d-1} \right) \cdot \frac{(p-1)d}{s^{d+1}} \\
& = K^2(d-1) s^{d-2} + \frac{ ((p-1)d +1) (p-1)}{s^{d+2}} - \frac{ K (d-1)(2-p)}{s^2}. \\
\end{align*}
This expression is $s^{-d-2}$ times a quadratic polynomial in $s^d$ and with easy calculation we see that it is nonnegative in $[0,K^{-1}]$ if and only if $p \geq \frac {4d}{3d+1}$.
\end{proof}
For the study of the Keller-Segel case we will use the following functional $\mathcal G$:

$$\mathcal G(\rho):=-\frac{\chi}{2} \int |\nabla h[\rho]|^2dx=-\frac{\chi}{2} \int h[\rho] d\rho,$$
where $\chi>0$ is a given constant and $h[\rho]$ is the only solution of 
$$\begin{cases}-\Delta h=\rho &\mbox{in } \Omega,\\
			h=0 &\mbox{on } \partial\Omega.\end{cases}$$
Note the negative sign before the integral in the definition of $\mathcal G$.
As we already did in Section 2, it is not difficult to check that we have
$$\frac{\delta\mathcal G}{\delta \rho}=-\chi h[\rho].$$

\begin{proposition}\label{prop:smoothing} Let $\eta$ be a probability measure and $F$ be a convex (but not necessarily smooth) function satisfying the McCann condition and let $\mathcal H(\rho)= \frac 1q\int_{\Omega} \rho^q \, dx$ with $q>d/2$. Then, let $\delta>0$ and consider $\rho \in Prox_{ \Ent + \delta \mathcal H+ \mathcal{G}}^{\tau} (\eta)$:  if $\Omega$ is convex we have
\begin{equation}\label{eqn:stimaKS} \int_{\Omega} F(\eta) \, dx \geq \int_{\Omega} F(\rho) \, dx - \tau\chi \int_{\Omega} [\rho F'_+(\rho) - F(\rho)] \rho \, dx + \tau  \int_{\Omega} |\nabla \rho|^2 F_{ac}''(\rho) \, dx, \end{equation}
Where $F''_{ac}$ is the absolutely continuous part of the derivative of $F'$ and $F'_+$ is the right derivative of $F$.
\end{proposition}

\begin{proof} Since $\rho \in L^q$  with $q > d/2$ we have that $h[\rho]\in W^{2,q}$ is bounded and HÃ¶lder continuous. Looking at the optimality condition
$$\log \rho +\delta  \rho^{q-1} = c- \varphi/\tau + \chi h[\rho]$$
we deduce  that $\rho$ is also H\"oder continuous, and bounded from above and below. As a consequence, by elliptic regularity, $h[\rho]$ is a $C^2$ function. Then $\rho$ has the same regularity as the worse between $\varphi$ and $h[\rho]$, and in particular $\rho$ is Lipschitz continuous. Now let us assume for a while that $F$ is convex and $C^2$: we start from \eqref{eqn:stimaKS} and replace $u[\rho]$ with $\delta \rho^{q-1}-\chi h[\rho]$. This provides

\begin{eqnarray*}\int_{\Omega} F(\eta) \, dx \geq \int_{\Omega} F(\rho) \, dx &+& \tau \int_{\Omega} F''(\rho)| \nabla\rho|^2dx\\&+& \delta\tau\int_\Omega\rho F''(\rho) \nabla\rho\cdot\nabla(\rho^{p-1})\,dx-\chi\tau \int_\Omega\rho F''(\rho) \nabla\rho\cdot\nabla h[\rho]\,dx. \end{eqnarray*}

Noting that $s\mapsto sF''(s)$ is the derivative of $s\mapsto sF'(s)-F(s)$, we can integrate by parts the last term, thus obtaining 
\begin{eqnarray*}
\int_\Omega\rho F''(\rho) \nabla\rho\cdot\nabla h[\rho]\,dx&=&-\int_\Omega (\rho F'(\rho)-F(\rho)) \Delta h[\rho]\,dx+\int_{\partial\Omega} (\rho F'(\rho)-F(\rho)) \nabla h[\rho]\cdot \nu\, d \sigma\\&\leq& \int_\Omega (\rho F'(\rho)-F(\rho))\rho\,dx.
\end{eqnarray*}
In the above inequality, $d\sigma$ denotes the uniform $(d-1)$-measure on the boundary $\partial\Omega$ (and the terms we integrate on the boundary make sense, because of regularity). Moreover, we used $\nabla h[\rho]\cdot \nu\leq 0$ (as a consequence of the positivity of $h[\rho]$ with its Dirichlet boundary condition on $\partial\Omega$) and $\rho F'(\rho)-F(\rho)\geq 0$ (a onsequence of the convexity of $F$ together with $F(0)=0$). Using this information, and the positivity of $\nabla\rho\cdot\nabla(\rho^{p-1})$, we get
\begin{equation}\label{estimate with FC2}\int_{\Omega} F(\eta) \, dx \geq \int_{\Omega} F(\rho) \, dx + \tau \int_{\Omega} F''(\rho)| \nabla\rho|^2-\chi\tau \int_\Omega (\rho F'(\rho)-F(\rho))\rho\,dx.
\end{equation}
%
%
Now let us consider any convex function $F$ (also not smooth) and let us approximate it by smooth convex functions satisfying McCann conditions. In order to do this, instead of directly approximating $F$ by convolution, we approximate by convolution $s\mapsto s^dF(s^{-d})$ and obtain some corresponding $F_\ep$ which satisfy $F_\ep(0)=0$, $F_{\ep} \to F$ uniformly $F'_{\ep} \to F'$ at any differentiability point of $F$ (and $\limsup F'_{\ep} \leq F'_+$ at every non-differentiability point, where $F'_+$ stands for the right-derivative of $F$), and $F''_{\ep} \to F''_{ac}$ almost everywhere, where $F''_{ac}$ is the absolutely continuous part of $F''$. This implies $|\nabla \rho|^2F_\ep''(\rho)\to |\nabla \rho|^2F_{ac}''(\rho)$ since the convergence holds on a.e. level set of $\rho$, and we have $|\nabla\rho|=0$ a.e. on $\{x\in \Omega\,:\, \rho(x)\in A\}$, where $A$ is the set of values on which we do not have the convergence $F''_{\ep} \to F''_{ac}$. Finally, we note that $F'_{\ep}(\rho), F'(\rho)$ are bounded since $ \rho$ is bounded from above and from below. \\
In particular we can pass to the limit in \eqref{estimate with FC2} using Fatou's lemma, thus obtaining

$$\int_{\Omega} F(\eta) \, dx \geq \int_{\Omega} F(\rho) \, dx + \tau \int_{\Omega} F''_{ac}(\rho)| \nabla\rho|^2dx-\chi\tau \int_\Omega (\rho F'_+(\rho)-F(\rho))\rho\,dx,$$
which proves the claim.
\end{proof}

The following estimates requires $d=2$ and $\chi$ sufficiently small.

Indeed, we first use the following fact, which we already mentioned in Section 2: for $\chi\leq 8\pi$, if $d=2$, the functional $\mathcal F:=\Ent + \mathcal{G}$ is bounded from below; hence, if $\chi<8\pi$, then there exist constants $A,B>0$ such that $\Ent(\rho)\leq A+B(\Ent(\rho)+\mathcal G(\rho))$ is true for every $\rho\in\pical(\Omega)$.

Moreover, another point of the proof where we use $d=2$ is an inequality where we use the BV norm to estimate the $L^2$ norm of a given function, which is a two dimensional fact.  

\begin{theorem}\label{lpestimate} Let  $\Omega\subset\R^2$ be a convex set and $\eta \in \mathscr{P}(\Omega) \cap L^p(\Omega)$ be given. For every positive number $e_0$ there is a $D_1=D_1(e_0,p)>0$ such that whenever $\mathcal F(\eta):=\Ent(\eta)+\mathcal G(\eta)\leq e_0$, for $K \geq K(e_0,p)$ there exists $ \rho \in  Prox^{\tau}_{ \Ent+ \mathcal{G}} ( \eta)$ with 

\begin{equation}\int_{\Omega} F_{p,K}(\rho) \, dx \leq \int_{\Omega} F_{p,K}(\eta)\, dx+  \tau D_1. \end{equation}
In particular $\rho \in L^p(\Omega)$; furthermore, denoting $\rho_0= \eta$ and $\rho_{n+1} \in Prox^{\tau}_{ \Ent+ \mathcal{G}} ( \rho_n)$ a measure selected so that the above estimate applies, we have $\| \rho_n \|^p_p \leq \| \rho_0 \|_p^p +(1+n\tau )D_2$ with $D_2$ again depending only on $e$ and $p$. 

Turning to the $L^\infty$ norm, we also have $\| \rho_n \|_{\infty} \leq  (1+n\tau)C (e, \| \rho_0\|_\infty)$.
\end{theorem}

\begin{proof} Let us consider $\rho_{\delta} = Prox^{\tau}_{\Ent + \delta H + \mathcal{G}} (\eta)$ as in Proposition \ref{prop:smoothing}. Then, we use $F=F_{p,K}$ with $K=k^{p- \frac{ d-1}d}$ in \eqref{eqn:stimaKS}. First note that we have $F_{p,K}>0$ if and only if $s>k$ as well as
$$0\leq \rho(F'_{p,K})_+(\rho)-F_{p,K}(\rho))\leq p\rho^p\ind_{\rho\geq k}, \quad (F''_{p,K})_{ac}(\rho)\geq p(p-1)\rho^{p-2}\ind_{\rho\geq k}.$$
This allows to write the inequality
\begin{equation}\label{eqn:stimagiusta} \int_{\Omega} F_{p,K}(\eta) \, dx-  \int_{\Omega} F_{p,K} (\rho_{\delta}) \, dx \geq - \tau \chi p\int_{\rho_\delta\geq k}\rho_\delta^{p+1} \, dx+  \tau p(p-1)\int_{\rho_\delta\geq k}|\nabla\rho_\delta|^2\rho_\delta^{p-2}\, dx,\end{equation}
Consider a constant $c_1=c_1(\Omega)$such that we have $|Du|(\Omega) \geq c_1 \| u\|_2$ for every function $u$ satisfying $ 2|\{ |u| >0\}| \leq |\Omega|$. Then, whenever $ 2|\{\rho \geq k\}|\leq \Omega$ (note that $| \Omega| \geq 2/k$ is enough) and $p>-1$, we have: 

\begin{align*} \int_{\rho \geq k} \rho\, dx \int_{\rho \geq k} \!| \nabla \rho|^2 \rho^{p-2} dx& \geq \left( \int_{\rho \geq k} | \nabla \rho | \rho^{\frac {p-1}2}\, dx \right)^2  = \frac 4{(p+1)^2} \left( \int_{\Omega} | \nabla (\rho^{\frac{p+1}2} - k^{\frac{p+1}{2}} )_+ | \, dx \right)^2 \\
& \geq  \frac{4c_1^2}{(p+1)^2} \int_{\Omega} (\rho^{\frac {p+1}2}-k^{\frac{p+1}2})_+^{2}\, dx \\
\end{align*}

Now we can use the inequality (valid for $p>0$)
$$\rho^{p+1} \leq \begin{cases} \frac {(\rho^{\frac {p+1}2}-k^{\frac{p+1}2})_+^{2}}{(1-2^{-(p+1)/2p})^2}  \qquad &\text{ if }\rho \geq 2^{1/p}k \\ 2k^p \rho & \text{ if }0\leq \rho \leq 2^{1/p}k. \end{cases}$$ 
Using $\frac{ p+1}{2p} > \frac 12$, we have $1/(1-2^{-\frac {p+1}{2p}})^2 \leq (1- 1/\sqrt{2})^{-2}= 6+4\sqrt{2} \leq 12$, and we find 
$$ \int_{\Omega} \rho^{p+1}\, dx \leq   2k^p\int_{\Omega} \rho \, dx + 12 \int_{\Omega} (\rho^{\frac {p+1}2}-k^{\frac{p+1}2})_+^{2} \, dx$$
$$ 12 \int_{\Omega} (\rho^{\frac {p+1}2}-k^{\frac{p+1}2})_+^{2} \, dx\geq \int_{\Omega} \rho^{p+1} \,dx -   2k^p.$$
In particular we find that
$$\int_{\rho_\delta \geq k} | \nabla \rho_\delta|^2 \rho_\delta^{p-2}\, dx \geq \frac{c_1^2}{3(p+1)^2\int_{\rho_\delta \geq k} \rho_\delta\, dx } \left( \int_{\Omega} \rho_\delta^{p+1} \, dx- 2k^p \right). $$ 
Now, it is sufficient to find $k $ such that 
$$\int_{\rho_{\delta} \geq k} \rho_{\delta}\, dx \leq \alpha(p):= \frac{(p-1)c_1^2}{3(p+1)^2\chi}$$
in order to obtain

$$\int_{\Omega} F_{p,K}(\eta) \, dx-  \int_{\Omega} F_{p,K} (\rho_{\delta}) \, dx 
\geq -2\tau\chi p k^p, $$
which would give the first part of the claim.

In order to estimate  $\int_{\rho_{\delta} \geq k} \rho_{\delta}$ we just use \begin{align*}
\int \rho_{\delta} | \log \rho_\delta|\leq \Ent(\rho_\delta)+|\Omega|e^{-1} & \leq A+|\Omega|e^{-1}+ B\mathcal F(\rho_{\delta})    \\
& \leq A+|\Omega|e^{-1}+ B\min_\rho\left\{  \mathcal{F}(\rho) + \delta \mathcal H(\rho)+\frac{W_2^2(\rho,\eta)}{2\tau}\right\}.
\end{align*}
By $\Gamma$-convergence, for $\delta$ sufficiently small we have
\begin{eqnarray*}
\min_\rho\left\{ \mathcal{F}(\rho) + \delta \mathcal H(\rho)+\frac{W_2^2(\rho,\eta)}{2\tau}\right\}&\leq& \min_\rho\left\{ \mathcal{F}(\rho) + \frac{W_2^2(\rho,\eta)}{2\tau}\right\}+1\\&\leq&   \mathcal{F}(\eta) +1\leq e_0+1.
\end{eqnarray*}

Then we can choose $k$ looking at
$$ \int_{\rho_{\delta} \geq k} \rho_{\delta} \leq \frac 1{ \log k } \int_{\rho_{\delta} \geq k}   \rho_{\delta} | \log \rho_{\delta}| \leq \frac { A+|\Omega|e^{-1}+B(e_0+1)}{\log k }.$$
It is then enough to choose $k$ large enough depending on $e_0$ and $p$, and in particular we choose
$$k=k(e_0,p)\quad \mbox{ with }\quad \frac { A+|\Omega|e^{-1}+B(e_0+1)}{\log k(e_0,p) }=  \frac{(p-1)c_1^2}{3(p+1)^2\chi}.$$
The value of $K(e_0,p)$ is defined accordingly, and then we find, for $K\geq K(e_0,p)$:
$$ \int_{\Omega} F_{p,K}(\rho_{\delta})\, dx \leq \int_{\Omega} F_{p,K}(\eta) \, dx + \tau D_1.$$

Now we let $\delta\to 0$: up to a subsequence we have $\rho_{\delta} \rightharpoonup \rho$ and the limit $\rho$ satisfies the same inequality; moreover we have also that $\rho \in Prox_{ \Ent+ \mathcal{F}}^{\tau}(\eta)$ thanks to the $\Gamma$-convergence of the functionals. For the global in time estimate we can iterate the previous result thanks to the fact that $\mathcal F(\rho_n)$ is decreasing in $n$, in order to get 
$$ \int_{\Omega} F_{p,K}(\rho_n) \, dx \leq  \int_{\Omega} F_{p,K}(\rho_0) \, dx + n \tau D_1. $$
Then we can use the  inequalities 
$$-k^{p-1} \rho + \rho^p \leq (\rho^p -k^{p-1} \rho)_+ \leq F_{p,K}(\rho) \leq \rho^p,$$ 
(remember $K=k^{p-(d-1)/d}$) to conclude 
\begin{align*} \int_{\Omega} | \rho_n|^p \, dx &\leq \int_{\Omega} F_{p,K}(\rho_n) \,dx + k(e_0,p)^{p-1} \leq \int_{\Omega} | \rho_0|^p \,dx + n \tau D_1 + k(e_0,p)^{p-1}   \\ & \leq \int_{\Omega} | \rho_0|^p \,dx + ( 1+ \chi n \tau )2pk(e_0,p)^p, \end{align*}
where in the last inequality we use the dependence of $D_1$ in terms of $k$ and we suppose $k\geq 1$.

For the $L^{\infty}$ estimate we cannot simply pass to the limit the $L^p$ inequality we just found since $k(e_0,p) \to \infty$ as $p \to \infty$. However we can iterate the procedure, finding better estimates for $k(e_0,p)$ using the fact that we have some explicit bounds on $\| \rho_m \|_p$: let us fix $n$ and let us consider $T=n\tau$.  We will consider iteratively $p_i=2^{i}+1$. Choosing always $k(e_0,p) \geq \| \rho_0\|_{\infty}$ we find that 
\begin{equation}\label{eq:it1} D(p_i) = \sup_{m \leq n}\|  \rho_m \|_{p_i}   \leq ((2+\chi T) 2p_i)^{1/p_i} k(e_0,p_i); \end{equation}
For the iterative step, we have
$$ \int_{\rho \geq k} \rho_{\delta} \leq \frac 1{ k^{p_i-1} } \int_{\rho \geq k}   \rho_{\delta}^{p_i} \, dx\leq \frac { D(p_i)^{p_i}}{k^{p_i-1} }; $$
in particular it is sufficient to choose
$$ k(e_0,p_{i+1}) = \left(\frac{ D(p_i)^{p_i}}{\alpha(p_i)}\right)^{\frac 1{p_i-1}};$$
now we can use $\alpha(p_i) \geq C/p_i \geq C 2^{-i} $ (where $C$ also depends on $\chi$) and \eqref{eq:it1} to obtain
$$ k(e_0,p_{i+1}) \leq \left(C2^{1+2i}(2+\chi T)\right)^{\frac 1{2^{i}}} k(e_0,p_i)^{1+ \frac 1{2^i}}.$$
From here we can derive a uniform bound for $k(e_0,p_i)$: indeed, defining $c_{i+1} = \prod_{j=2}^i (1+\frac 1{2^j}) $ and writing $k_i:=k(e_0,p_i)$, we have 
\begin{align*}
k_{i+1}^{ \frac 1{c_{i+1}}} & \leq  \left(C2^{1+2i}(2+\chi T)\right)^{\frac 1{2^{i}c_{i+1}}} k_i^{ \frac 1{c_i}} \\
& \leq k_2 \prod_{j=2}^{i+1}  \left(C2^{1+2j}(2+\chi T)\right)^{\frac 1{2^{j}}}\\
& \leq D k_2\sqrt{2+\chi T} .
\end{align*}
In particular , using \eqref{eq:it1} together with the last estimate and the fact that $c_i \leq 2$, we can say that we have
$$\|\rho_n \|_{\infty} \leq \lim_{i \to \infty} D(p_i) \leq \lim_{i \to \infty} k(e_0,p_i) \leq D (2+\chi T) \max \{k(e_0,5),\| \rho_0\|_{\infty}\} ^2.$$
\end{proof}

The above estimate is the discrete counterpart of a well-known result studied in continuous time, which can be found for instance in \cite{Per-book}, and proven in \cite{JL} and \cite{DolPer} (more precisely: \cite{JL}  showed that equi-integrability of $\rho$ is enough to propagate in time the estimates on the $L^p$ norms, and \cite{DolPer} found the sharp condition to bound the entropy of the solution, and hence provide equi-integrability).

We note that in the above $L^p$ estimate we obtain a linear growth in time of the $L^p$ norm raised to the power $p$, i.e. on $\int \rho^pdx$, so that the norm itself has much slower growth. In this concern the estimate on the $L^\infty$ norm is most likely not sharp, as it is the norm itself which grows linearly.

Concerning the $L^\infty$ estimates, we remind that other $L^\infty$ bounds have been found on a (perturbed) JKO scheme in \cite{CarSanKS}, but those bounds always explose in finite time (at time $T=1/||\rho_0||_{\infty}$). On the other hand, they have the advantage that they are true for any form of diffusion, and that they require no condition on $\chi$, nor on the dimension.

\begin{remark}With similar but more tedious calculation it is possibile also to get hypercontractivity estimates (improvement in time of the summability exponent) in the JKO setting:  this could potentially weaken the integrability requirement on $\rho_0$ in Theorem~\ref{thm:main}, but then we  would need a different analysis for the first steps, where the integrability assumption $\rho \in L^r$ is still not satisfied. For this reason we don't want to pursue this direction here, but it would be for sure interesting.
\end{remark}

In a similar but different spirit we also mention that it is possible to obtain $L^q$ estimates in time and space starting from $L^p$ assumptions on $\rho_0$, as it is done in Lemma 2.11 of \cite{KwonMes}, but we do not investigate this question here since we decided to concentrate on bounds which are not integrated in time but derive from a decreasing behavior from a step of the JKO scheme to the next one.
%
%
%
%

\section{Sobolev estimates}

In this section we pass to the core of the paper, i.e. the higher order estimates. The goal will be to obtain results comparable to those of Proposition \ref{summary warmup}, but for norms involving the gradient of $\rho$.
%
%
%
%

\begin{lemma}\label{generalH} Let us consider a functional $\mathcal{G} : \mathscr{P} (\Omega) \to \R \cup \{\infty\}$ and a probability $\eta\in \pical(\Omega)$. Take $\rho \in Prox_{ \Ent +\mathcal{G}}^{\tau} (\eta)$ and set $u[\rho]:=\delta\mathcal G/\delta\rho$. Suppose that $\Omega$ is convex and let $H:\R^d\to\R$ be a radial convex function. Set 
$$Z_\rho:=\frac{\nabla\rho}{\rho}+\nabla u[\rho],\quad Z_\eta:=\frac{\nabla\eta}{\eta}+\nabla u[\eta]$$
and call $\varphi$ and $\psi$ the Kantorovich potentials for the transport from $\rho$ to $\eta$, with $T=id-\nabla\varphi$ the optimal transport map from $\rho$ to $\eta$.
Then we have
$$\int H(Z_\eta)d\eta\geq \int H(Z_\rho)d\rho+\int \nabla H\left(\frac{\nabla\varphi}{\tau}\right)\cdot \left(\nabla u[\rho]-\nabla u[\eta]\circ T\right)d\rho.$$
\end{lemma}
\begin{proof}
We start from the fact that $H$ is radial and hence even, and that it is convex:
$$\int H(Z_\eta)d\eta=\int H(-Z_\eta)d\eta\geq \int H\left(\frac{-\nabla\psi}{\tau}\right)d\eta+ \int \nabla H\left(\frac{-\nabla\psi}{\tau}\right)\cdot\left(-Z_\eta+\frac{\nabla\psi}{\tau}\right)d\eta.$$
We look at the different parts of the right-hand sice.
First we use we use $\eta=T_\#\rho$ and $-\nabla\psi\circ T=\nabla\varphi$, together with the optimality condition $Z_\rho+\frac{\nabla \varphi}{\tau}=0$ and again the fact that $H$ is even, in order to get
$$\int H\left(\frac{-\nabla\psi}{\tau}\right)d\eta=\int H\left(\frac{\nabla\varphi}{\tau}\right)d\rho= \int H(Z_\rho)d\rho.$$
Using again $\eta=T_\#\rho$ and $-\nabla\psi\circ T=\nabla\varphi$, we obtain
\begin{eqnarray*}
\int \nabla H\left(\frac{-\nabla\psi}{\tau}\right)\cdot\frac{\nabla\psi}{\tau}d\eta&=&-\int \nabla H\left(\frac{\nabla\varphi}{\tau}\right)\cdot\frac{\nabla\varphi}{\tau}d\rho=\int \nabla H\left(\frac{\nabla\varphi}{\tau}\right)\cdot Z_\rho d\rho\\
&=&\int \nabla H\left(\frac{\nabla\varphi}{\tau}\right)\cdot \nabla\rho \,dx+\int \nabla H\left(\frac{\nabla\varphi}{\tau}\right)\cdot \nabla u[\rho]d\rho.
\end{eqnarray*}
We now pass to the part involving $Z_\eta$, and write
\begin{eqnarray*}
 \int \nabla H\left(\frac{-\nabla\psi}{\tau}\right)\cdot\left(-Z_\eta\right)d\eta&=&\int \nabla H\left(\frac{\nabla\psi}{\tau}\right)\cdot\left(Z_\eta\right)d\eta\\
 &=&\int \nabla H\left(\frac{\nabla\psi}{\tau}\right)\cdot\nabla\eta\,dx+\int \nabla H\left(\frac{\nabla\psi}{\tau}\right)\cdot\nabla u[\eta]\, d\eta\\
 &=&\int \nabla H\left(\frac{\nabla\psi}{\tau}\right)\cdot\nabla\eta\,dx-\int \nabla H\left(\frac{\nabla\varphi}{\tau}\right)\cdot\nabla u[\eta]\circ T\,d\rho.
\end{eqnarray*}
Summing up all the terms, and using the five-gradient inequalities (which requires $H$ to be radial in order to handle the boundary terms)
$$\int \nabla H\left(\frac{\nabla\psi}{\tau}\right)\cdot\nabla\eta \, dx+\int \nabla H\left(\frac{\nabla\varphi}{\tau}\right)\cdot \nabla\rho \, dx\geq 0,$$
we obtain the desired result.
\end{proof}

The quantities of the form $\int H(Z_\rho)d\rho$ will be crucial for the Sobolev regularity of the solutions of the JKO scheme. We will then often note $J_{(p)}(\rho):= \int H(Z_\rho)d\rho$ when $H(z)=|z|^p$, without explicit reference to the term $u[\rho]$, which will be clear from the context.

\subsection{Fokker-Planck and aggregation}

We will see some consequences of Lemma \ref{generalH}, starting from the easiest case, i.e. the purely linear Fokker-Planck case: $u[\rho]=V$ and $\mathcal G(\rho):=\int V\,d\rho$.

\begin{proposition}\label{decrHZ} Let us consider $\mathcal G(\rho):=\int Vd\rho$ and $\rho \in Prox_{ \Ent + \mathcal{G}}^{\tau} (\eta)$. Suppose that $\Omega$ is convex and that $V$ is semi-convex, i.e. $D^2V\geq \lambda I$.
Then, if $\lambda=0$ and $H:\R^d\to\R$ is an arbitrary radial convex function we have 
$$\int H(Z_\rho)d\rho\leq \int H(Z_\eta)d\eta.$$
If $\lambda>0$ and $H(0)=0$, we also have
$$(1+\lambda \tau)\int H(Z_\rho)d\rho\leq \int H(Z_\eta)d\eta.$$
For $\lambda<0$, if $H$ satisfies $\nabla H(z)\cdot z\leq C(H(z)+1)$ then 
$$(1-|\lambda| C\tau)\int H(Z_\rho)d\rho\leq \int H(Z_\eta)d\eta+C\tau.$$
\end{proposition}

\begin{proof} All these results are just a consequence of Lemma \ref{generalH}. They can be obtained if one estimates the term 
$\int \nabla H\left(\frac{\nabla\varphi}{\tau}\right)\cdot \left(\nabla u[\rho]-\nabla u[\eta]\circ T\right)d\rho.$
First, note that since $H$ is radial, the vectors $\nabla H\left(\nabla\varphi/\tau\right)$ and $\nabla\varphi$ are parallel and oriented in the same direction. We also use $u[\rho]=u[\eta]=V$ and the assumptions on $V$. Indeed we have
$$(\nabla V(x)-\nabla V(T(x)))\cdot \nabla\varphi(x)=(\nabla V(x)-\nabla V(x-\nabla\varphi(x)))\cdot \nabla\varphi(x)\geq \lambda |\nabla\varphi(x)|^2,$$
thanks to the $\lambda$-convexity of $V$. In the case $\lambda=0$ this is enough to obtain the claim.

For $\lambda>0$, we write
$$\int\! \nabla H\!\left(\!\frac{\nabla\varphi}{\tau\!}\right)\cdot \left(\nabla V-\nabla V\circ T\right)d\rho=\int \frac{\left| \nabla H\left(\frac{\nabla\varphi}{\tau}\right)\right|}{\left| \frac{\nabla\varphi}{\tau}\right|}(\nabla V-\nabla V\circ T)\cdot \nabla\varphi\, d\rho\geq \lambda\tau\int H\left(\!\frac{\nabla\varphi}{\tau}\!\right) d\rho,$$
where we used the inequality $|\nabla H(z)|\geq H(z)/|z|$ which is valid for radial convex functions with $H(0)=0$ and the same estimate due to the $\lambda$-convexity of $V$ as above. This allows to prove the second part of the claim.

In the case $\lambda<0$, the estimate is similar, but since we estimate the scalar product $(\nabla V(x)-\nabla V(T(x)))\cdot \nabla\varphi(x)$ from below with $\lambda |\nabla\varphi(x)|^2,$ which is negative, we need to estimate $|\nabla H(z)|$ from above, and for this we use our assumption on $H$ (which is essentially an assumption of polynomial growth for $H$; note that, $H$ being radial, we have $\nabla H(z)\cdot z=|\nabla H(z)||z|$).
\end{proof}

\begin{remark} As we underlined in the introduction, the above result is a time-discrete translation of a suitable intergal version of a well-known estimate in the Bakry-Emery theory (again, we refer for instance to \cite{BakGenLed}). Indeed, the time-continuous equation satisifed by $\rho$ when taking the gradient flow of $\Ent+\mathcal G$ is $\partial_t\rho-\Delta\rho-\nabla\cdot(\rho\nabla V)=0$. If one defines $u=\rho e^{V}$ then $u$ satisfies the drift-diffusion PDE $\partial_t u=\Delta u -\nabla V\cdot\nabla u$. If we call $P_t$ the semigroup associated with this PDE, the celebrated Bakry-Emery estimates provide $|\nabla (P_t f)|\leq e^{-\lambda t}P_t(|\nabla f|)$ when $D^2V\geq \lambda I$. Taking $H$ convex and radially increasing, and using this inequality, for instance for $\lambda=0$, together with the convexity of the function $(s,y)\mapsto H(y/s)s$, one can prove
$$\int H\left(\frac{\nabla (P_t f)}{P_t f}\right)P_t \, de^{-V}\leq \int H\left(\frac{P_t(|\nabla f|)}{P_t f}\right)P_t f\, de^{-V}\leq \int H\left(\frac{|\nabla f|}{f}\right)f\, de^{-V},$$
which can be seen to be equivalent to the result of Proposition \ref{decrHZ} since
$$\frac{\nabla (P_t f)}{P_t f}=\frac{\nabla u}{u}=\frac{(\nabla\rho+\rho\nabla V)e^V}{\rho e^V}=Z_\rho.$$
\end{remark}

As a consequence, we obtain the following information on the JKO scheme
\begin{proposition}\label{prop5.3}
If $V$ is $\lambda$-convex and Lipschitz, the JKO scheme for the Fokker-Plack equation preserves the following bounds
\begin{itemize}
\item if $\lambda\geq 0$, if $\log\rho_0$ is Lipschitz continuous, then $\log\rho^\tau_n$ is also Lipschitz continuous, with bounded Lipschitz constant, and $\Lip(\log\rho^\tau_n+V)$ decreases in time (in $n$);
\item if  $\rho_0^{1/p}\in W^{1,p}(\Omega)$, then $(\rho^\tau_n)^{1/p}$ is bounded in $W^{1,p}(\Omega)$ independently of $\tau$ and $k$;
\item if $\rho_0\in BV(\Omega)$, then $\rho^\tau_n$ is bounded in $BV(\Omega)$ independently of $\tau$ and $n$;
\item if $\rho_0\in W^{1,1}(\Omega)$, then all the densities $\rho^\tau_n$ belong to a weakly-compact subset of $W^{1,1}(\Omega)$.
\end{itemize}
Moreover, if $\lambda>0$ then the gradients of the functions $(\rho^\tau_n e^V)^{1/p}$ converge exponentially fast to $0$ in $L^{p}(e^{-V})$, uniformly with respect to $\tau$, and hence the functions $(\rho^\tau_n e^V)^{1/p}$ converge in $W^{1,p}(e^{-V})$ to a constant.
\end{proposition}

\begin{proof}
In the case $\lambda\geq 0$, for the first part of the statement, take a measure $\eta$ and  $ \rho \in  Prox^{\tau}_{ \Ent+ \mathcal{G}} ( \eta)$. Let us suppose $\Lip(\log\eta+V)\leq L$ and use as a function $H$ the convex indicator function of $\overline{B(0,L)}$. From $\int H(Z_\eta)d\eta=0$ we deduce that we also have $\int H(Z_\rho)d\rho=0$. This means $|\nabla (\log\rho+V)|\leq L$ a.e. on $\{\rho>0\}$. Yet, we know from the optimality conditions that $\rho$ is a continuous density which is bounded away from $0$ since $\log\rho=C-V-\varphi/\tau$, hence we get $\Lip(\log\rho+V)\leq L$. This can be iterated along the JKO scheme thus obtaining the first part ot the statement.

For the second part of the statement, we use $H(z)=|z|^p$ and
$$||\rho^{1/p}||^p_{W^{1,p}}=c(p)\int \left(\rho^{1/p-1}|\nabla\rho|\right)^p+\int \rho\leq c(p)\int H(Z_\rho)d\rho+C,$$
where we used the boundedness of $\nabla V$. Since $\int H(Z_{\rho^\tau_n})d\rho^\tau_n$ decreases with $n$ (if $\lambda\geq 0$) or at least its growth is exponentially controlled (if $\lambda<0$), then $(\rho^\tau_n)^{1/p}$ is bounded in $W^{1/p}$.

The third part of the statement is proven in a similar way, using $p=1$. Indeed, given  $\eta\in BV(\Omega)$ and  $ \rho \in  Prox^{\tau}_{ \Ent+ \mathcal{G}} ( \eta)$, we can approximate $\eta$ with smoother densities $\eta_j$ with $||\nabla\eta_j+\eta_j\nabla V||\to ||\nabla\eta+\eta\nabla V||$ (the norm being taken in the space of vector measures). For each $j$ we have a measure $\rho_j \in Prox^{\tau}_{ \Ent+ \mathcal{G}} ( \eta_j)$, which is Lipschitz continuous and satisfies 
$$||\nabla\rho_j+\rho_j\nabla V||=\int \left|\frac{\nabla \rho_j}{\rho_j}+\nabla V\right|d\rho_j\leq ||\nabla\eta_j+\eta_j\nabla V||,$$ 
and, passing to the limit in $j$, we get
$$ ||\nabla\rho+\rho\nabla V||\leq ||\nabla\eta+\eta\nabla V||.$$
This proves that $ ||\nabla\rho^\tau_n+\rho^\tau_n\nabla V||$ is decreasing, and hence $||\nabla\rho^\tau_n||$ stays bounded.

In what concerns the $W^{1,1}$ estimate, we use a convex and superlinear function $H$ such that $\int H(Z_{\rho_0})d\rho_0<\infty$ (which exists since $\nabla \rho_0\in L^1$ implies $\nabla \rho_0/\rho_0+\nabla V\in L^1(\rho_0)$ and we know that $L^1$ functions are also integrable when composed with a suitable superlinear function, which can be taken convex). The results of Proposition \ref{decrHZ} allow then to keep the same integrability of the gradient along the iterations of the JKO scheme: this is easy if $\lambda\geq 0$, while for $\lambda<0$ we just need to note that $H$ can be taken superlinear but with polynomial growth (actually, its growth can be taken as close to linear as we want), and hence we can apply the last claim in Proposition \ref{decrHZ}. This guarantees equi-integrability for $\nabla\rho^\tau_n$ and hence the claim. 

We are now left to consider the behavior for $n\to\infty$ in the case $\lambda>0$. In this case we have exponential decay of the quantity
$J_{(p)}(\rho^\tau_n):=\int H(Z_{\rho^\tau_n})d\rho^\tau_n$ for $H(z)=|z|^p$. We can then observe that we have
$$\int \left|\frac{\nabla \rho}{\rho}+\nabla V\right|^pd\rho=\int \left|\nabla\log (\rho e^V) \right|^pe^{-V}d(\rho e^V) = c \int \left|\nabla\left((\rho e^V)^{1/p}\right)\right|^p d(e^{-V}).$$
This last result provides a sort of rate of convergence to the steady state of the Fokker-Planck equation $\rho=e^{-V}$. For the $W^{1,p}$ convergence we only have to use an appropriate local Sobolev inequality and exploit uniform integrability. \end{proof}

\begin{remark} In seeking quantitative convergence for the JKO in $W^{1,p}$, one has to exploit also the geodesic convexity of $\mathcal{F}_U : \rho \mapsto \int  U\left( \frac {\rho} {e^{-V}} \right) e^{-V} \, dx $ whenever $\lambda>0$ and $U$ satisfies appropriate generalization of the McCann condition; however this goes beyond the scope of this work.

It is however easier to treat the time continuous case because, whenever $U$ is convex, $\mathcal{F}_U$ is actually decreasing along the evolution (see \cite{IacPatSan} for a similar computation)
\end{remark}

\begin{remark}
The bound on the Lipschit constant could have been obtained as a limit on the $L^p$ norms for $p\to\infty$. Indeed, for $\lambda\geq 0$ we can also easily obtain
$$\int |Z_\rho|^pd\rho\leq \int |Z_\eta|^pd\eta$$
which, raising to power $1/p$ and sending $p\to \infty$ also gives a bound on the $L^\infty$ norm of $Z_\rho$, and hence on the Lipschitz constant of $\log\rho+V$. On the other hand, the approach with $H(z)=|z|^p$ is interesting for $\lambda<0$, as it provides
$$\int |Z_\rho|^pd\rho\leq (1-|\lambda|p\tau)^{-1}\int |Z_\eta|^pd\eta.$$
This estimate provides exponential bounds on $\int |Z_{\rho_t}|^pd\rho_t$, i.e. $\int |Z_{\rho_t}|^pd\rho_t\leq e^{|\lambda|pt}\int |Z_{\rho_0}|^pd\rho_0$. By taking the power $1/p$ and the limit as $p\to\infty$, one gets $||Z_{\rho_t}||_{\infty}\leq e^{|\lambda|t}||Z_{\rho_0}||_{\infty}$. Yet, this last computation can only be performed in continous time. More precisely, we first need to send $\tau\to 0$ and then $p\to \infty$. Indeed, if we first send $p\to\infty$ while $\tau>0$ is fixed, we would get $1-|\lambda|p<0$ which prevents any interesting estimate to be obtained. 
\end{remark}

In the next remark we use the following notation: when a vector $z$ and an exponent $\alpha>0$ are given, by $z^\alpha$ we mean $|z|^{\alpha-1}z$ (if $z\neq 0$, and $0$ if $z=0$), i.e. a vecotr whose norm is $|v|^\alpha$ and the direction is the same as that of $v$.

\begin{remark} \label{rmk:cascade}
The estimate with $H(z)=|z|^p$ when $V$ is $\lambda$-convex with $\lambda<0$ can also be concluded in a different way when $p<2$. Indeed, we can we can use
$$ \int_{\Omega} ( \nabla V-\nabla V\circ T ) \cdot \left( \frac{ x- T(x) }{\tau} \right)^{p-1} \, d \rho \geq  \frac \lambda{ \tau^{p-1}} \int_{\Omega} | x-T(x)|^p \, d \eta \geq \lambda \frac {W_2(\eta, \rho)^{p}} {{\tau}^{p-1}},$$
where in the last passage we use $p<2$. In particular if $\rho_n$ is the sequence generated by the JKO scheme we have, by induction,
   $$ J_{(p)}(\rho_{n+1})\leq J_{(p)}(\rho_n) +  p |\lambda| \sum_{i=k}^{n} \tau \left( \frac {W_2(\rho_i, \rho_{i+1})}{\tau} \right)^p.$$
Now we can use the Holder inequality and, together with $W_2^2(\rho_n,\rho_{n+1})\leq 2\tau (\mathcal F(\rho_n)-\mathcal F(\rho_{n+1}))$, we obtain
\begin{align*}J_{(p)}(\rho_{n+1} ) & \leq J_{(p)}(\rho_n) +  p|\lambda| \bigl((n-k)\tau\bigr)^{1-p/2}\left( \sum_{i=k}^{n} \tau \left( \frac {W_2(\rho_i, \rho_{i+1})}{\tau} \right)^2 \right)^{p/2} \\
& \leq J_{(p)}(\rho_n) + p |\lambda |(t-s)^{1-p/2} \left( 2\mathcal{F} (\rho_{n+1}) - 2\mathcal{F}(\rho_n) \right)^{p/2}
\end{align*}
Hence, we deduce that  $J_{(p)}(\rho_{t})$ is locally bounded in time, and grows sublinearly as $t\to\infty$.
\end{remark}

We also want to consider the case where $V$ depends on $\rho$ via a smooth convolution kernel. This is typical in aggregation equations.
We consider the following case
\begin{equation*}\label{Ginter}
\mathcal G(\rho):=\int Vd\rho+\frac 12 \int\int W(x-y)d\rho(x)d\rho(y)
\end{equation*}
for an interaction potential $W:\R^d\to\R$ which is supposed to be even ($W(z)=W(-z)$) and $C^{1,1}$. This last assumption is very demanding and non-optimal, but allows for a simple presentation of the estimates. The Keller-Segel case that we will see later is in some sense obtained from a singular interaction potential ($W(z)=\log|z|$ in dimension $d=2$, when in the whole space),and will be treated in details, but in a different way.
We will set
$$J_{(p)}(\rho):=\int \left|\frac{\nabla\rho}{\rho}+\nabla V+\nabla W*\rho\right|^pd\rho,$$
i.e. $J_{(p)}(\rho)=\int H(Z_\rho)d\rho$, where $u[\rho]=V+W*\rho$. For simplicity, in the notation $J_{(p)}$, we are omitting the dependence on $V$ and $W$.

\begin{proposition}\label{decrHZ} Let us consider $\mathcal{G}$ as in \eqref{Ginter}, and $\rho \in Prox_{ \Ent + \mathcal{G}}^{\tau} (\eta)$. Suppose that $\Omega$ is convex, that $V$ is semi-convex, i.e. $D^2V\geq \lambda I$, and that $W$ is $C^{1,1}$, with $\Lip(\nabla W)=\mu$. Then, we have 
$$(1+p(\lambda-  2 \mu)\tau)J_{(p)}(\rho)\leq J_{(p)}(\eta).$$
\end{proposition}

\begin{proof}
The starting point is, of course, the result of Lemma \ref{generalH} applied to $H(z)=|z|^p$. This gives
\begin{eqnarray*}
J_{(p)}(\eta)&\leq& J_{(p)}(\rho)\\&&+p\!\int \!\left(\frac{x-T(x)}{\tau}\right)^{p-1}\!\!\cdot \left(\nabla V(x)\!-\!\nabla V(T(x))+(\nabla W\!*\!\rho)(x)\!-\!(\nabla W\!*\!\eta)(T(x))\right)d\rho(x),
\end{eqnarray*}
where by $v^{p-1}$, when $v$ is a vector (here $v=(x-T(x))/\tau$), we mean $|v|^{p-2}v$.

Using the same argument as in Remark \ref{rmk:cascade} we can obtain
$$\int \left(\frac{x-T(x)}{\tau}\right)^{p-1}\cdot \left(\nabla V(x)-\nabla V(T(x))\right)d\rho(x)\geq \tau\lambda \int\left|\frac{x-T(x)}{\tau}\right|^{p}d\rho(x)=\tau \lambda J_{(p)}(\rho),$$
as well as 
$$\int \left(\frac{x-T(x)}{\tau}\right)^{p-1}\cdot \left((\nabla W*\eta)(x)-(\nabla W*\eta) (T(x))\right)d\rho(x)\geq -\tau\mu J_{(p)}(\rho),$$
since the function $W*\eta$ is $(-\mu)$-convex.

We are left to estimate the remaining term 
$$\int \left(\frac{x-T(x)}{\tau}\right)^{p-1}\cdot \left((\nabla W*\eta)(x)-(\nabla W*\rho) (x)\right)d\rho(x).$$
This term will be bounded in absolute value, and we first note that we have
$$|(\nabla W*\rho)(x)-(\nabla W*\eta)(x)|=\left|\int \nabla W(x-y)d(\rho-\eta)(y)\right|\leq \mu W_1(\rho,\eta),$$
since $y\mapsto  \nabla W(x-y)$ is $\mu$-Lipschitz for every $x$. We also use
$$W_1(\rho,\eta)\leq\int |x-T(x)|d\rho=\tau \int \frac{|x-T(x)|}{\tau}d\rho\leq \tau\left( \int \left|\frac{x-T(x)}{\tau}\right|^pd\rho\right)^{1/p}=\tau J_{(p)}(\rho)^{1/p}.$$
Then we have

\begin{eqnarray*}
\left|\int \left(\frac{x-T(x)}{\tau}\right)^{p-1}\!\!\cdot \left((\nabla W*(\eta-\rho)) (x)\right)d\rho(x)\right|
&\leq& \mu\tau J_{(p)}(\rho)^{1/p}\int \left|\frac{x-T(x)}{\tau}\right|^{p-1}d\rho(x)\\&\leq& \mu\tau J_{(p)}(\rho)^{1/p}J_{(p)}(\rho)^{(p-1)/p}=\mu\tau J_{(p)}(\rho).
\end{eqnarray*}
Putting all the results together provides the claim.
\end{proof}

Let us note that in the above result, in order to obtain an estimate which could be iterated, we needed to replace $\nabla W*\eta$ with $\nabla W*\rho$, and hence we used the Lipschitz behavior of $\nabla W$: for this estimate, lower bounds on $D^2W$ (as we required on $V$) were not enough. We also note that the same kind of exponential asymptotic behavior providing convergence to the steady state as in the last point of Proposition \ref{prop5.3} could be obtained, provided $\lambda>2\mu$, but these computations will not be detailed.

\subsection{ Keller-Segel case}

We come back to the case $\mathcal{G}(\rho)=- \frac \chi 2 \int h[\rho] \, d \rho $, where $-\Delta h[\rho]= \rho$ in $\Omega$ with Dirichlet boundary conditions. In this case we have that the first variation of $\mathcal{G}$ is $u[\rho]=- \chi h[\rho]$. The main goal of this section will be to have an estimate of 

\begin{equation}\label{JpKS} J_{(p)}(\rho) = \int_{\Omega} \left| \frac { \nabla \rho}{\rho}  -\chi \nabla h[\rho] \right|^p \, d\rho. 
\end{equation}

As we will see, in order to deal some of the error terms, we will need to estimate $J_{(p)}$ with $W_2(\rho,\eta)$, and use $p<2$. Moreover, we will also need an apriori bound on the $L^r$ norm of $\rho$ and $\eta$, for an exponent $r$ depending on $p$. All these restrictions mainly due to the fact that $-h[\rho]$ does not necessarily satisfies the semiconvexity assumptions that we usually used to handle remainder terms.

Thanks to Lemma \ref{generalH} we have 

\begin{align}\label{1stestJpKS}J_{(p)}(\rho) \leq J_{(p)}(\eta) &+ p \chi \int_{\Omega} ( \nabla h[\rho]  -  \nabla h[\eta] \circ T ) \cdot \left( \frac{ x- T(x) }{\tau} \right)^{p-1} \, d\rho \notag  \\
 = J_{(p)}(\eta)  &+ p \chi \int_{\Omega} (\nabla h[\rho] - \nabla h[\eta] ) \cdot \left( \frac{ x- T(x) }{\tau} \right)^{p-1} \, d\rho \notag \\&- p \chi   \int_{\Omega} (\nabla h[\eta] \circ T- \nabla h[\eta] ) \cdot \left( \frac{ x- T(x) }{\tau} \right)^{p-1} \, d\rho 
\end{align}

In order to treat the two remainder terms, we state a general comparison result between Sobolev dual norms and Wasserstein distances (Exercise 35 in \cite{OTAM}).

\begin{proposition}\label{prop:sobone} Let $\rho, \eta \in \mathscr{P}_2(\Omega)$ be two absolutely continuous measures. Then, supposing that $\| \rho \|_r, \|\eta\|_r \leq C$ with $ \frac 1q + \frac 1r +\frac 1p = 1 + \frac 1rq$, for every  $\varphi \in C^{1}(\Omega)$ we have
$$ \int \varphi\,d (\rho- \eta) \leq \| \nabla \varphi \|_p \cdot C^{1/q'} \cdot W_q ( \eta, \rho). $$
In particular we have $\| \rho- \eta\|_{(W^{1,p})^*(\Omega)} \leq \sqrt{ \max\{ \| \rho\|_r , \| \eta \|_r\}} W_2( \eta, \rho)$ for $r=p/(p-2)$.
\end{proposition}
 The following lemma is very classical, and can be found, for instance, in \cite{Giusti}, Theorem~10.15 (where the case $p\geq 2$ is treated, for $p< 2$ one can argue then by duality) 
\begin{lemma}\label{lem:sobone2} Let $\Omega$ be a bounded convex domain, and $f \in{(W^{1,p}_0)^*(\Omega)}$ be given. Denoting by $h[f]$ the unique solution in $W^{1,p'}_0(\Omega)$ (with $p'=p/(p-1)$ the dual exponent to $p$) of $-\Delta h=f$, there exists a constant $C>0$, depending only on the dimension, on $p$ and possibly on $\Omega$, such that
$$  \| \nabla h[f] \|_{p'} \leq  C\| f \|_{(W^{1,p})^*(\Omega)} . $$
\end{lemma}

\begin{lemma}\label{lem:estimates} Given $p\in (1,2)$, set $r= \frac {4-p}{2-p}$. Let us assume $\rho, \eta \in \mathscr{P}(\Omega) \cap L^r(\Omega)$ with $\Omega$ convex and let us denote by $T$ the optimal map between $\rho$ and $\eta$. Then there exists a constant $C$, only depending on $\Omega$, $p$, and $d$, such that
  $$  \int_{\Omega} (\nabla h[\eta] - \nabla h[\eta] \circ T) \cdot \left( \frac{ x- T(x) }{\tau} \right)^{p-1} \, d \rho\leq C\tau \max\{ \| \rho\|_r^r , \| \eta \|_r^r \}  +  \frac{W_2^2(\eta, \rho)}{\tau}  $$
 $$\int_{\Omega} (\nabla h[\rho] - \nabla h[\eta] ) \cdot \left( \frac{ x- T(x) }{\tau} \right)^{p-1} \, d \rho \leq C\tau \max\{ \| \rho\|_r^r , \| \eta \|_r^r \}  +  \frac{W_2^2(\eta, \rho)}{\tau}  $$
\end{lemma}
\begin{proof} We begin with the first inequality: if we set $T_t(x)=x+t (T(x)-x)$ then we know that $(T_t)_{\sharp} \rho := \rho_t \in \mathscr{P}(\Omega)$ is the Wasserstein geodesic from $\rho$ to $\eta$, and the displacement convexity properties of the $L^r$ norms imply that we have $\| \rho_t\|_r \leq \max\{ \| \eta\|_r , \| \rho \|_r\}=:M$. We have

\begin{multline*} \int_{\Omega} (\nabla h[\eta] - \nabla h[\eta] \circ T) \cdot \left( \frac{ x- T(x) }{\tau} \right)^{p-1} \, d \rho \\
= \int_{\Omega} \int_0^1  ( x-T(x)) \cdot (D^2 h[\eta] (T_t(x))) \cdot \left( \frac{ x- T(x) }{\tau} \right)^{p-1} \, dt \, d \rho \\
 \leq  \int_{[0,1] \times \Omega}  \tau^{1-p/2} |D^2 h[\eta] (T_t(x))| \cdot \frac{ |x- T(x)|^p }{\tau^{p/2}}  \, dt \, d\rho \\
 \leq  \tau^{\frac 1q} \left( \int_0^1 \int_{\Omega}   |D^2 h[\eta] (T_t(x))|^q \, d\rho  \, d t \right)^{\frac 1q} \cdot \left(  \int_{\Omega} \frac{ |x- T(x)|^2 }{\tau }\, d\rho   \right)^{\frac p2} \\
 = \tau^{\frac 1q}  \left( \int_0^1 \int_{\Omega}| D^2 h[\eta] (x) |^q \rho_t (x) \, dx \, dt \right)^{\frac 1q} \cdot \left( \frac {W_2(\eta, \rho) ^2}{\tau} \right)^{\frac p2},
\end{multline*}
where $q$ denotes here the dual exponent of $2/p$, i.e. $q= \frac 2{2-p} = r-1$. Using H\"older inequality with exponents $\frac{q+1}q$ and $q+1$ and then the classical estimate $\| D^2 h[\rho] \|_r \leq C \| \Delta h[\rho]\|_r =C \| \rho\|_r$ we obtain 
$$\int | D^2 h[\rho]|^q \eta_t \leq \| D_2 h[\rho]\|_{q+1}^q \| \eta_t\| \leq  C M^{q+1}.$$ 
Using this estimate and eventually  Young inequality with exponents $q$ and $2/p$ we obtain

 \begin{align*}  \left( \int_0^1 \int_{\Omega} \tau | D^2 h[\rho] (x) |^q \eta_t (x) \, dx \, dt \right)^{\frac 1q} \cdot \left( \frac {W_2(\eta, \rho) ^2}{\tau} \right)^{\frac p2} &\leq \left(\tau^{\frac 1q}  C M^{\frac {q+1}q}  \right) \cdot \left( \frac {W_2(\eta, \rho) ^2}{\tau} \right)^{\frac p2} \\
 & \leq \tau  C M^r +  \frac {W_2(\eta, \rho) ^2}{\tau}
 \end{align*}
Now we can pass to the second inequality. We perform directly a Holder inequality with exponents$ \frac 2{3-p}$ and $\frac 2{p-1}$, and then a Holder inequality with exponents $\frac {4-p}2$ and $\frac{ 4-p}{2-p}$:

\begin{align*} \int_{\Omega} (\nabla h[\rho] - \nabla h[\eta] ) \cdot \left( \frac{ x- T(x) }{\tau} \right)^{p-1} \, d \rho &\leq \frac { 1}{\tau^{p-1}} \left( \int_{\Omega} | \nabla h[\rho-\eta] |^{\frac 2{3-p}} \, d \rho \right)^{\frac {3-p}2} \cdot \left( \int_{\Omega} |T(x)-x |^2 \, d \eta \right)^{\frac {p-1}2} \\
& \leq \frac 1{\tau^{p-1}} \| \nabla h[\eta- \rho] \|_{\frac {4-p}{3-p} } \| \eta \|_r^{\frac {3-p}2} W_2(\rho, \eta)^{p-1}
\end{align*}
We then use Proposition \ref{prop:sobone} and Lemma \ref{lem:sobone2} in order to write 
$$\| \nabla h[\eta- \rho] \|_{\frac {4-p}{3-p} } \leq C \|\eta -\rho\|_{ (W^{1,4-p}(\Omega))^*} \leq C \sqrt{M} W_2(\eta, \rho).$$ 
Using this and then a Young inequality, we obtain
$$  \frac 1{\tau^{p-1}} \| \nabla h[\eta- \rho] \|_{\frac {4-p}{3-p} } \| \eta \|_r^{\frac {3-p}2} W_2(\rho, \eta)^{p-1} \leq \frac C{\tau^{p-1}} M^{\frac {4-p}2} W_2(\rho, \eta)^p \leq \tau C M^r + \frac { W_2(\rho,\eta)^2 }{\tau}.$$
\end{proof}

We can then collect all the previous results to obtain the following estimate.
\begin{theorem}\label{thm:main}
Let $(\rho^\tau_n)_n$ be a sequence obtained from the JKO scheme for the Keller-Segel functional $\mathcal F:=\Ent+\mathcal G$, and let $J_{(p)}$ be defined as in \eqref{JpKS}, for $p<2$. Then we have
\begin{equation}\label{itestJpKS}
J_{(p)}(\rho^\tau_{n+1})+\mathcal F(\rho^\tau_{n+1})\leq J_{(p)}(\rho^\tau_{n})+\mathcal F(\rho^\tau_{n})+C\tau \max\{ \| \rho^\tau_n\|_r^r , \| \rho^\tau_{n+1}\|_r^r \}.
\end{equation}
Hence, if $J_{(p)}(\rho_0)<\infty$, if $\mathcal F$ is bounded from below and if $\| \rho^\tau_n\|_r^r$ stays bounded along iterations, then $(\rho^\tau_n)^{1/p}$ is bounded in $W^{1,p}$. In particular, for every $\alpha>0$, in the  two-dimensional Keller-Segel model with $\chi<8\pi$, when $(\rho_0)^{1/p}\in L^\infty\cap W^{1,p}$ there exists a solution $\rho_t$ which satisfie a bound of the form $J_{(p)}(\rho_t)\leq C+Ct^{1+\alpha}$, (the constant $C$ depending on $\alpha$, on $p$ and on the initial datum). If we do not suppose $(\rho_0)^{1/p}\in W^{1,p}$, then the same bound will be true for $t\geq t_0>0$, with a constant depending on $t_0$.
\end{theorem}

\begin{proof} The iterated estimate \eqref{itestJpKS} is just a consequence of \eqref{1stestJpKS} and of Lemma \ref{lem:estimates}, together with 
$$\frac{W_2^2(\rho^\tau_n,\rho^\tau_{n+1})}{\tau}\leq \mathcal F(\rho^\tau_n)-\mathcal F(\rho^\tau_{n+1}).$$
If $\mathcal F$ is bounded from below and the $L^r$ norm from above, we then obtain a bound on $J_{(p)}(\rho^\tau_{n})$ and then we use 
$$||(\rho)^{1/p}||_{W^{1,p}}\leq C+CJ_{(p)}(\rho)+C\int |\nabla h[\rho]|^pd\rho.$$
Using $\rho\in L^r$, we just need to bound $ \nabla h[\rho]$ in $L^{pr'}$ in order to obtain the desired Sobolev bound. From elliptic regularity, we have $h[\rho]\in W^{2,r}$ and hence $\nabla h[\rho]\in L^{r^*}$, with $r^*=dr/(d-r)$, if $r<d$ (the case $r\geq d$ is easy, since in this case $\nabla h[\rho]$ belongs to all Lebesgue spaces). We just need to check $r^*\geq pr'$, which is true using $r\geq 3$ and $p<2$.

In order to apply this estimate to the two-dimensional Keller-Segel model with $\chi<8\pi$, we first note that in this case $\mathcal F$ is bounded from below (this would also be true for $\chi=8\pi$). We just need to bound the $L^r$ norm of the solution, and for this we use the estimates of Theorem \ref{lpestimate}. We obtain $||\rho_t||_q\leq C(1+t)^{1/q}$ for arbitrary $q>1$, since $\rho_{0}\in L^\infty$. Taking $q=r/\alpha$ one gets $||\rho_t||_r^r\leq C(1+t)^{\alpha}$ and the conclusion follows iterating  \eqref{itestJpKS}.

If we do not suppose $J_{(p)}(\rho_0)<\infty$, then we can use the dissipation of the energy $\mathcal F$ itself. Indeed, this dissipation provides $\mathcal F(\rho_0)=\int_0^{t_0} J_{(2)}(\rho_t)dt+\mathcal F(\rho_{t_0})$, which means that we can assume $J_{(2)}(\rho_{t_1})\leq \mathcal F(\rho_0)/t_0$ for some $t_1\in (0,t_0)$. This provides finiteness of $J_{(p)}(\rho_{t_1})$ for every $p<2$. We then start a JKO scheme from $\rho_{t_1}$.\end{proof}
%

{\bf Acknowledgements} The work started when both the authors were at Laboratoire de Mathématiques d'Orsay, the first author being supported as a post-doc researcher by the ANR project ISOTACE (ANR-12-MONU-0013). This work is also supported by the Agence nationale de la recherche via the project ANR-16-CE40-0014 - MAGA - Monge-Ampère et Géométrie Algorithmique.

\end{document}